\documentclass[10pt,reqno]{amsart}
\usepackage[hypertex]{hyperref}

\usepackage{amsmath,amsthm}
\usepackage{amssymb}
\usepackage{euscript}
\usepackage[all,tips]{xy}

\numberwithin{equation}{subsection}

\newcommand{\sqsp}{\renewcommand{\baselinestretch}{1.15}\tiny\normalsize}

\newcommand{\relphantom[1]}{\mathrel{\phantom{#1}}}

\newcommand{\nicearrow}{\SelectTips{cm}{10}}

\newtheorem{theorem}[subsection]{Theorem}
\newtheorem{lemma}[subsection]{Lemma}

\newtheorem{corollary}[subsection]{Corollary}

\theoremstyle{definition}
\newtheorem{definition}[subsection]{Definition}
\newtheorem{example}[subsection]{Example}

\newcommand{\bk}{\mathbf{k}}

\newcommand{\bracket}{[,...\,,]}
\newcommand{\nproduct}{(,...,)}

\DeclareMathOperator{\Hom}{Hom}

\begin{document}

\title{A Hom-associative analogue of $n$-ary Hom-Nambu algebras}
\author{Donald Yau}

\begin{abstract}
It is shown that every $n$-ary totally Hom-associative algebra with equal twisting maps yields an $n$-ary Hom-Nambu algebra via an $n$-ary version of the commutator bracket.  The class of $n$-ary totally Hom-associative algebras is shown to be closed under twisting by self-weak morphisms.  Every multiplicative $n$-ary totally Hom-associative algebra yields a sequence of multiplicative totally Hom-associative algebras of exponentially higher arities.  Under suitable conditions, an $n$-ary totally Hom-associative algebra gives an $(n-k)$-ary totally Hom-associative algebra.
\end{abstract}

\keywords{Hom-Nambu algebra, totally Hom-associative algebra, $n$-commutator bracket.}

\subjclass[2000]{17A40, 17A42, 17B81}

\address{Department of Mathematics\\
    The Ohio State University at Newark\\
    1179 University Drive\\
    Newark, OH 43055, USA}
\email{dyau@math.ohio-state.edu}

\date{\today}
\maketitle

\sqsp

\section{Introduction}

Nambu algebras play an important role in physics.  For example, Nambu mechanics \cite{nambu,tak} involves $n$-ary Nambu algebras \cite{filippov,plet}.  In these $n$-ary Nambu algebras, the $n$-ary compositions are derivations, a property called the $n$-ary Nambu identity (a.k.a. Filippov identity), which generalizes the Jacobi identity.  Lie triple systems \cite{jacobson1,lister}, which are ternary Nambu algebras with two further properties, can be used to solve the Yang-Baxter equation \cite{okubo}.  An $n$-ary Nambu algebra whose product is anti-symmetric is called an $n$-ary Nambu-Lie algebra (a.k.a. Filippov algebra).   Ternary Nambu-Lie algebras appear in the work of Bagger and Lambert \cite{bl} and many others on M-theory.  Other applications of $n$-ary Nambu algebras in physics are discussed in, e.g., \cite{ai,okubo,plet}.

Generalizations of $n$-ary Nambu(-Lie) algebras, called $n$-ary Hom-Nambu(-Lie) algebras, were introduced by Ataguema, Makhlouf, and Silvestrov in \cite{ams} .  In an $n$-ary Hom-Nambu(-Lie) algebra, the $n$-ary Nambu identity is replaced by the $n$-ary Hom-Nambu identity (see Definition \ref{def:homnambulie}), which involves $n-1$ linear twisting maps.  These twisting maps can be thought of as additional degrees of freedom or as deformation parameters.  An $n$-ary Nambu(-Lie) algebra can be regarded as an $n$-ary Hom-Nambu(-Lie) algebra in which the twisting maps are all equal to the identity map.  Hom-Nambu(-Lie) algebras are interesting even when the underlying algebras are Nambu(-Lie) algebras.  In fact, certain ternary Nambu-Lie algebras can be regarded as ternary Hom-Nambu-Lie algebras with non-identity twisting maps \cite{ams}.

Binary Hom-Nambu-Lie algebras are called Hom-Lie algebras, which originated in \cite{hls} in the study of $q$-deformations of the Witt and the Virasoro algebras.  The binary Hom-Nambu identity in this case is called the Hom-Jacobi identity \cite{ms}.  The associative counterparts of Hom-Lie algebras are Hom-associative algebra \cite{ms}.  Hom-Lie algebras are to Hom-associative algebras as Lie algebras are to associative algebras \cite{ms,yau}.

Let us recall some properties of $n$-ary Hom-Nambu(-Lie) algebras.  It is shown in \cite{ams} that $n$-ary Nambu(-Lie) algebras can be twisted along self-morphisms to yield $n$-ary Hom-Nambu(-Lie) algebras.  This twisting construction of $n$-ary Hom-Nambu(-Lie) algebras is a generalization of a result about $G$-Hom-associative algebras due to the author \cite{yau2}.  It is shown in \cite{ams1} that ternary Virasoro-Witt algebras can be $q$-deformed into ternary Hom-Nambu-Lie algebras.  It is shown in \cite{ams2} that a ternary Hom-Nambu-Lie algebra can be obtained from a Hom-Lie algebra together with a compatible linear map and a trace function.

Further properties of $n$-ary Hom-Nambu(-Lie) algebras were established by the author in \cite{yau13,yau14}.  A unique feature of Hom-type algebras is that they are closed under twisting by suitably defined self-morphisms.  In particular, it is shown in \cite{yau13} that the category of $n$-ary Hom-Nambu(-Lie) algebras is closed under twisting by self-weak morphisms.  Starting with an $n$-ary Nambu(-Lie) algebra, this closure property reduces to the twisting construction for $n$-ary Hom-Nambu(-Lie) algebras in \cite{ams}.  Moreover, it is proved in \cite{yau13} that every multiplicative $n$-ary Hom-Nambu algebra yields a sequence of Hom-Nambu algebras of exponentially higher arities.  It is also shown in \cite{yau13} that, under suitable conditions, an $n$-ary Hom-Nambu(-Lie) algebra reduces to an $(n-k)$-ary Hom-Nambu(-Lie) algebra.

Hom-Jordan and Hom-Lie triple systems were defined in \cite{yau13} as Hom-type generalizations of Jordan and Lie triple systems \cite{jacobson1,lister,meyberg}.  A Hom-Lie triple system is automatically a ternary Hom-Nambu algebra, but it is usually not a ternary Hom-Nambu-Lie algebra because its ternary product is not assumed to be anti-symmetric.  It is proved in \cite{yau13} that Hom-Lie triple systems, and hence ternary Hom-Nambu algebras, can be obtained from Hom-Jordan triple systems, ternary totally Hom-associative algebras \cite{ams}, multiplicative Hom-Lie algebras, and Hom-associative algebras.

Furthermore, it is proved in \cite{yau14} that multiplicative Hom-Jordan algebras \cite{yau12} have underlying Hom-Jordan triple systems, and hence also ternary Hom-Nambu algebras.  Combined with results from \cite{yau13}, this implies that every multiplicative Hom-Jordan algebra gives rise to a sequence of Hom-Nambu algebras of arities $2^{k+1} + 1$.  As in the classical case, a major source of Hom-Jordan algebras is the class of Hom-alternative algebras, which were defined in \cite{mak}.  It is proved in \cite{yau12} that multiplicative Hom-alternative algebras are Hom-Jordan admissible.  Therefore, every multiplicative Hom-alternative algebra also gives rise to a sequence of Hom-Nambu algebras of arities $2^{k+1} + 1$.  Finally, the class of $n$-ary Hom-Nambu-Lie algebras is extended to the class of $n$-ary Hom-Maltsev algebras in \cite{yau14}.


The main purpose of this paper is to study a Hom-associative analogue of $n$-ary Hom-Nambu algebras.  The basic motivation is that the $n$-ary Hom-Nambu identity is an $n$-ary version of the Hom-Jacobi identity.  As is well-known, the commutator bracket of an associative algebra satisfies the Jacobi identity.  Likewise, the commutator bracket of a Hom-associative algebra satisfies the Hom-Jacobi identity \cite{ms}.  Therefore, it is natural to ask the following question.
\begin{quote}
Is there an $n$-ary version of a Hom-associative algebra that gives rise to an $n$-ary Hom-Nambu algebra via an $n$-ary version of the commutator bracket?
\end{quote}
One main result of this paper is an affirmative answer to this question when the twisting maps are equal (Theorem \ref{thm:commutator}).  The relevant Hom-associative type objects are the $n$-ary totally Hom-associative algebras defined in \cite{ams}, which generalize $n$-ary totally associative algebras.  The relevant commutator bracket is what we call the $n$-commutator bracket (Definition \ref{def:nbracket}), which involves $2^{n-1}$ terms.  Restricting to the special case where all the twisting maps are equal to the identity map, this result implies that every $n$-ary totally associative algebra yields an $n$-ary Nambu algebra via the $n$-commutator bracket.

Low-dimensional cases of the $n$-commutator bracket have been used elsewhere.  In particular, the $2$-commutator bracket is the usual commutator.  The $3$-commutator bracket was used by the author in \cite{yau13} to show that a ternary totally Hom-associative algebra with equal twisting maps yields a ternary Hom-Nambu algebra.  We should point out that our $n$-commutator bracket is not the same as the totally anti-symmetrized $n$-ary commutator in \cite{ap,bp}.  Moreover, the $n$-commutator bracket is not anti-symmetric when $n \geq 3$.  Therefore, the $n$-ary Hom-Nambu algebras arising from $n$-ary totally Hom-associative algebras via the $n$-commutator bracket are usually not $n$-ary Hom-Nambu-Lie algebras.

A description of the rest of this paper follows.

In section \ref{sec:homass} we observe that the class of $n$-ary totally Hom-associative algebras is closed under twisting by self-weak morphisms (Theorem \ref{thm:closure}).  The corresponding closure property for $n$-ary Hom-Nambu(-Lie) and $n$-ary Hom-Maltsev algebras can be found in \cite{yau13} and \cite{yau14}, respectively.  A special case of Theorem \ref{thm:closure} says that each multiplicative $n$-ary totally Hom-associative algebra gives rise to a sequence of multiplicative $n$-ary totally Hom-associative algebras by twisting along its own twisting map (Corollary \ref{cor2:closure}).  We obtain Theorem 3.6 in \cite{ams} as another special case of Theorem \ref{thm:closure}.  It says that $n$-ary totally associative algebras can be twisted along self-morphisms to yield multiplicative $n$-ary totally Hom-associative algebras  (Corollary \ref{cor3:closure}).  Section \ref{sec:homass} ends with several examples of $n$-ary totally Hom-associative algebras.

In section \ref{sec:higher} we study how totally Hom-associative algebras of different arities are related.  First, we show that every multiplicative $n$-ary totally Hom-associative algebra yields a sequence of multiplicative totally Hom-associative algebras of exponentially higher arities, namely, $2^k(n-1) + 1$ for $k \geq 0$ (Corollary \ref{cor1:higher}). The transition from Hom-associative algebras to ternary totally Hom-associative algebras was proved in \cite{yau13}.  One major difference between the $n = 2$ case in \cite{yau13} and the $n \geq 3$ case here is that the latter requires multiplicativity while the former does not.  Second, we show that under suitable conditions an $n$-ary totally Hom-associative algebra reduces to an $(n-k)$-ary totally Hom-associative algebra (Corollary \ref{cor1:lower}).  The corresponding results for Hom-Nambu algebras can be found in \cite{yau13}.

In section \ref{sec:homassnambu} we define the $n$-commutator and show that every $n$-ary totally Hom-associative algebra with equal twisting maps yields an $n$-ary Hom-Nambu algebra via the $n$-commutator (Theorem \ref{thm:commutator}).  The special case of this result when $n=2$ is Proposition 1.6 in \cite{ms}.  The special case when $n=3$ is Corollary 4.3 in \cite{yau13}.  In these low dimensional cases, due to the relatively small number of terms involved, the binary and ternary Hom-Nambu identities can be shown by a direct computation with all the terms written out.  In the general case, a more systematic argument is needed because of the large number of terms in the Hom-Nambu identity.  In fact, the $n$-ary Hom-Nambu identity for the $n$-commutator bracket involves $2^{2n-2}(n+1)$ terms.

\section{Totally Hom-associative algebras}
\label{sec:homass}

In this section we observe that the class of $n$-ary totally Hom-associative algebras is closed under twisting by self-weak morphisms.  Some examples of $n$-ary totally Hom-associative algebras are then given.

\subsection{Conventions}

Throughout this paper we work over a fixed field $\bk$ of characteristic $0$.  If $V$ is a $\bk$-module and $f \colon V \to V$ is a linear map, then $f^n$ denotes the composition of $n$ copies of $f$ with $f^0 = Id$.  For $i \leq j$, elements $x_i,\ldots,x_j \in V$ and maps $f,f_k,\ldots,f_l \colon V \to V$ with $j-i=l-k$, we adopt the abbreviations
\begin{equation}
\label{xij}
\begin{split}
x_{i,j} &= (x_i,x_{i+1},\ldots,x_j),\\
f(x_{i,j}) &= (f(x_i), f(x_{i+1}),\ldots, f(x_j)),\\
f_{k,l}(x_{i,j}) &= (f_k(x_i),f_{k+1}(x_{i+1}),\ldots,f_l(x_j)).
\end{split}
\end{equation}
For $i > j$, the symbols $x_{i,j}$, $f(x_{i,j})$, and $f_{k,l}(x_{i,j})$ denote the empty sequence.  For a bilinear map $\mu \colon V^{\otimes 2} \to V$, we often write $\mu(x,y)$ as the juxtaposition $xy$.

Let us begin with the following basic definitions.

\begin{definition}
\label{def:nhomalgebra}
Let $n \geq 2$ be an integer.
\begin{enumerate}
\item
An \textbf{$n$-ary Hom-algebra} $(V,\nproduct,\alpha)$ with $\alpha=(\alpha_1,\ldots,\alpha_{n-1})$ \cite{ams} consists of a $\bk$-module $V$, an $n$-linear map $\nproduct \colon V^{\otimes n} \to V$, and linear maps $\alpha_i \colon V \to V$ for $i = 1, \ldots , n-1$, called the \textbf{twisting maps}.
\item
An $n$-ary Hom-algebra $(V,\nproduct,\alpha)$ is said to be \textbf{multiplicative} if (i) the twisting maps are all equal, i.e., $\alpha_1 = \cdots = \alpha_{n-1} = \alpha$, and (ii) $\alpha \circ \nproduct = \nproduct \circ \alpha^{\otimes n}$.
\item
A \textbf{weak morphism} $f \colon V \to U$ of $n$-ary Hom-algebras is a linear map of the underlying $\bk$-modules such that $f \circ \nproduct_V = \nproduct_U \circ f^{\otimes n}$.  A \textbf{morphism} of $n$-ary Hom-algebras is a weak morphism such that $f \circ (\alpha_i)_V = (\alpha_i)_U \circ f$ for $i = 1, \ldots n-1$ \cite{ams}.
\end{enumerate}
\end{definition}

For an $n$-ary Hom-algebra $V$ and elements $x_1,\ldots,x_n \in V$, using the abbreviations in \eqref{xij}, the $n$-ary product $(x_1,\ldots,x_n)$ will often be denoted by $(x_{1,n})$ below.  We sometimes omit the commas in the $n$-ary product $\nproduct$.

An $n$-ary Hom-algebra $V$ in which all the twisting maps are equal, as in the multiplicative case, will be denoted by $(V,\nproduct,\alpha)$, where $\alpha$ is the common value of the twisting maps.  An \textbf{$n$-ary algebra} in the usual sense is a $\bk$-module $V$ with an $n$-linear map $\nproduct \colon V^{\otimes n} \to V$.  We consider an $n$-ary algebra $(V,\nproduct)$ also as an $n$-ary Hom-algebra $(V,\nproduct,Id)$ in which all $n-1$ twisting maps are the identity map.  Also, in this case a weak morphism is the same thing as a morphism, which agrees with the usual definition of a morphism of $n$-ary algebras.

Let us now recall the definition of an $n$-ary totally Hom-associative algebra from \cite{ams}.


\begin{definition}
\label{def:nhomass}
Let $(A,\nproduct,\alpha)$ be an $n$-ary Hom-algebra.
\begin{enumerate}
\item
For $i \in \{1,\ldots,n-1\}$ define the \textbf{$i$th Hom-associator} $as^i_A \colon A^{\otimes 2n-1} \to A$ to be the $(2n-1)$-linear map
\[
\begin{split}
as^i_A(a_{1,2n-1}) &= (\alpha_{1,i-1}(a_{1,i-1}), (a_{i,i+n-1}), \alpha_{i,n-1}(a_{i+n,2n-1}))\\
&\relphantom{} - (\alpha_{1,i}(a_{1,i}),(a_{i+1,i+n}), \alpha_{i+1,n-1}(a_{i+n+1,2n-1}))
\end{split}
\]
for $a_1,\ldots,a_{2n-1} \in A$.
\item
An \textbf{$n$-ary totally Hom-associative algebra} is an $n$-ary Hom-algebra $A$ that satisfies \textbf{total Hom-associativity}
\[
as^i_A = 0
\]
for all $i \in \{1,\ldots,n-1\}$.
\end{enumerate}
\end{definition}

An $n$-ary totally Hom-associative algebra with $\alpha_i = Id$ for all $i$ is called an \textbf{$n$-ary totally associative algebra}.  In this case, $as^i_A = 0$ is referred to as \textbf{total associativity}.

When $n=2$ total Hom-associativity means
\[
((a_1a_2)\alpha(a_2)) = (\alpha(a_1)(a_2a_3)),
\]
which is the defining identity for Hom-associative algebras \cite{ms}.  When $n=3$ total Hom-associativity means
\begin{equation}
\label{totalhomass3}
((a_1a_2a_3),\alpha_1(a_4),\alpha_2(a_5))
= (\alpha_1(a_1),(a_2a_3a_4),\alpha_2(a_5))
= (\alpha_1(a_1),\alpha_2(a_2),(a_3a_4a_5)).
\end{equation}
In particular, if both twisting maps $\alpha_1$ and $\alpha_2$ are equal to the identity map, then \eqref{totalhomass3} is the defining identity for ternary rings \cite{lister2}.

To see that the category of $n$-ary totally Hom-associative algebras is closed under twisting by self-weak morphisms, we need the following observations.

\begin{lemma}
\label{lem:closure}
Let $(A,\nproduct,\alpha)$ be an $n$-ary Hom-algebra and $\beta \colon A \to A$ be a weak morphism.  Consider the $n$-ary Hom-algebra
\begin{equation}
\label{abeta}
A_\beta = (A,\nproduct_\beta = \beta\nproduct,\beta\alpha = (\beta\alpha_1,\ldots,\beta\alpha_{n-1})).
\end{equation}
Then the following statements hold.
\begin{enumerate}
\item
$\beta^2 as_A^i = as_{A_\beta}^i$ for $i \in \{1,\ldots,n-1\}$.
\item
If $A$ is multiplicative and $\beta\alpha = \alpha\beta$, then $A_\beta$ is also multiplicative.
\end{enumerate}
\end{lemma}

\begin{proof}
Both assertions are immediate from the definitions.
\end{proof}


The desired closure property is now an immediate consequence of Lemma \ref{lem:closure}.

\begin{theorem}
\label{thm:closure}
Let $(A,\nproduct,\alpha)$ be an $n$-ary totally Hom-associative algebra and $\beta \colon A \to A$ be a weak morphism.  Then $A_\beta$ in \eqref{abeta} is also an $n$-ary totally Hom-associative algebra.  Moreover, if $A$ is multiplicative and $\beta\alpha = \alpha\beta$, then $A_\beta$ is also multiplicative.
\end{theorem}

Let us now discuss some special cases of Theorem \ref{thm:closure}. If $A$ is a multiplicative $n$-ary Hom-algebra, then its twisting map is a morphism (and hence a weak morphism) on $A$.  Therefore, we have the following special case of Theorem \ref{thm:closure}.

\begin{corollary}
\label{cor1:closure}
Let $(A,\nproduct,\alpha)$ be a multiplicative $n$-ary totally Hom-associative algebra.  Then
\[
A_\alpha = (A,\nproduct_\alpha=\alpha\nproduct,\alpha^2)
\]
is also a multiplicative $n$-ary totally Hom-associative algebra.
\end{corollary}

Iterating Corollary \ref{cor1:closure} we obtain the following result, which says that every multiplicative $n$-ary totally Hom-associative algebra gives rise to a sequence of derived $n$-ary totally Hom-associative algebras.

\begin{corollary}
\label{cor2:closure}
Let $(A,\nproduct,\alpha)$ be a multiplicative $n$-ary totally Hom-associative algebra.  Then
\[
A_k = (A,\nproduct_k=\alpha^{2^k-1}\nproduct,\alpha^{2^k})
\]
is also a multiplicative $n$-ary totally Hom-associative algebra for each $k \geq 0$.
\end{corollary}

On the other hand, if we set $\alpha_i = Id$ for all $i$ in Theorem \ref{thm:closure}, then we obtain the following twisting result, which is Theorem 3.6 in \cite{ams}.  It says that $n$-ary totally Hom-associative algebras can be obtained from $n$-ary totally associative algebras and their morphisms.

\begin{corollary}
\label{cor3:closure}
Let $(A,\nproduct)$ be an $n$-ary totally associative algebra and $\beta \colon A \to A$ be a morphism.  Then
\[
A_\beta = (A,\nproduct_\beta=\beta\nproduct,\beta)
\]
is a multiplicative $n$-ary totally Hom-associative algebra.
\end{corollary}


The rest of this section contains examples of $n$-ary totally Hom-associative algebras.

\begin{example}
\label{ex1}
Let $A$ be an associative algebra, $f \colon A \to A$ be an algebra morphism, and $\zeta \in \bk$ be a primitive $n$th root of unity.  Then the $\zeta$-eigenspace of $f$,
\begin{equation}
\label{afzeta}
A(f,\zeta) = \{a \in A \colon f(a) = \zeta a\},
\end{equation}
is an $(n+1)$-ary totally associative algebra under the $(n+1)$-ary product
\begin{equation}
\label{n1product}
(a_1,\ldots,a_{n+1}) = a_1\cdots a_{n+1}.
\end{equation}
Let $\alpha \colon A \to A$ be an algebra morphism such that $\alpha f = f\alpha$.  Then $\alpha$ restricts to a morphism of $(n+1)$-ary totally associative algebras on $A(f,\zeta)$.  By Corollary \ref{cor3:closure} there is a multiplicative $(n+1)$-ary totally Hom-associative algebra
\[
A(f,\zeta)_\alpha = (A(f,\zeta),\nproduct_\alpha,\alpha),
\]
where
\[
(a_1,\ldots,a_{n+1})_\alpha = \alpha(a_1\cdots a_{n+1})
\]
for all $a_i \in A$.
\qed
\end{example}

\begin{example}
\label{ex2}
Let $A$ be the associative algebra over $\bk$ consisting of polynomials in $r \geq 2$ associative variables $X_1,\ldots,X_r$ with $0$ constant term.  Fix an integer $n \geq 2$.  Let $A_n$ denote the submodule of $A$ spanned by the homogeneous polynomials of degrees $1 \pmod{n}$.  Note that $A_n$ is actually an eigenspace $A(f,\zeta)$ as in \eqref{afzeta}, where $\zeta \in \bk$ is a primitive $n$th root of unity and $f \colon A \to A$ is determined by
\[
f(X_i) = \zeta X_i
\]
for all $i$.  In any case, $A_n$ is an $(n+1)$-ary totally associative algebra with the $(n+1)$-ary product in \eqref{n1product}.  When $n=2$ the ternary totally associative algebra $A_3$ is an example in \cite{lister2} (p.47).

For each $i \in \{1,\ldots,r\}$ let $m_i \geq 1$ be an integer with $m_i \equiv 1 \pmod{n}$.  Then the map $\alpha \colon A \to A$ determined by
\[
\alpha(X_i) = X_i^{m_i}
\]
for all $i$ is an algebra morphism that commutes with $f$.  As in Example \ref{ex1} there is a multiplicative $(n+1)$-ary totally Hom-associative algebra $(A_n)_\alpha = (A_n,\nproduct_\alpha,\alpha)$.

Note that $(A_n,\nproduct_\alpha)$ is not totally associative because
\[
((\underbrace{X_1\cdots X_1}_{n+1})_\alpha \underbrace{X_2 \cdots X_2}_{n})_\alpha = X_1^{m_1^2(n+1)}X_2^{m_2n},
\]
whereas
\[
(\underbrace{X_1 \cdots X_1}_{n} (X_1\underbrace{X_2\cdots X_2}_{n})_\alpha)_\alpha = X_1^{m_1(m_1+n)}X_2^{m_2^2n}.
\]
They are not equal, provided $m_1 > 1$ or $m_2 > 1$.
\qed
\end{example}

\begin{example}
\label{ex3}
Fix an integer $n \geq 2$, and let $V_1,\ldots,V_n$ be $\bk$-modules.  Consider the direct sum
\[
A = \bigoplus_{i=1}^n \Hom(V_i,V_{i+1}),
\]
where $V_{n+1} \equiv V_1$.  A typical element in $A$ is written as $\oplus f_i$, where $f_i \in \Hom(V_i,V_{i+1})$ for $i \in \{1,\ldots,n\}$.   One can visualize the element $\oplus f_i \in A$ as the braid diagram
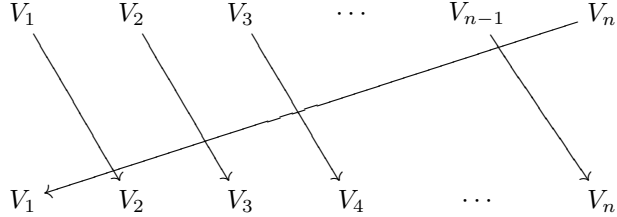
\[
\nicearrow
\xymatrix{
V_1 \ar[ddr] & V_2 \ar[ddr] & V_3 \ar[ddr] & \cdots & V_{n-1}\ar[ddr] & V_n \ar[ddlllll]\\
& & & & & \\
V_1 & V_2 & V_3 & V_4 & \cdots & V_n
}
\]
on $n$ strands.  Define an $(n+1)$-ary product on $A$ by
\[
\left(\oplus f_i^1,\ldots,\oplus f_i^{n+1}\right) = \oplus F_i,
\]
where
\[
F_i = \underbrace{f_i^{n+1} \cdots f_1^{n-i+2}}_{i} \underbrace{f_n^{n-i+1}\cdots f_{i+1}^2 f_i^1}_{n-i+1}
\]
for each $i$.  Pictorially, the $(n+1)$-ary product $\oplus F_i$ is represented by the vertical composition of $n+1$ braid diagrams.  With this $(n+1)$-ary product, $A$ becomes an $(n+1)$-ary totally associative algebra.  When $n=2$ the ternary totally associative algebra $A$ is an example in \cite{lister2} (p.46).

For each $i \in \{1,\ldots,n\}$, let $\gamma_i \colon V_i \to V_i$ be a linear automorphism.  Define the map $\alpha \colon A \to A$ by
\[
\alpha(\oplus f_i) = \oplus \gamma_{i+1}^{-1}f_i\gamma_i,
\]
where $\gamma_{n+1} \equiv \gamma_1$.  Then $\alpha$ is an automorphism of $(n+1)$-ary totally associative algebras.  By Corollary \ref{cor3:closure} there is a multiplicative $n$-ary totally Hom-associative algebra
\[
A_\alpha = (A,\nproduct_\alpha = \alpha\nproduct,\alpha).
\]
Moreover, $(A,\nproduct_\alpha)$ is in general not an $n$-ary totally associative algebra.  For example, suppose $f^i_i \in \Hom(V_i,V_{i+1})$ for $i \in \{1,\ldots,n\}$ and $f^{n+1}_1 \in \Hom(V_1,V_2)$.  Then
\[
((f^1_1,\ldots,f^n_n,f^{n+1}_1)_\alpha, f^2_2,\ldots, f^{n+1}_1)_\alpha = \gamma_2^{-1}f_1^{n+1}\cdots f^2_2\gamma_2^{-1}F_1\gamma_1^2,
\]
whereas
\[
(f_1^1,\ldots,f^n_n, (f_1^{n+1},f^2_2,\ldots, f^{n+1}_1)_\alpha)_\alpha = \gamma_2^{-2}f_1^{n+1}\cdots f^2_2f_1^{n+1}\gamma_1 f^n_n \cdots f^1_1 \gamma_1.
\]
They are not equal in general.
\qed
\end{example}

\section{Totally Hom-associative algebras of different arities}
\label{sec:higher}

There are two main results in this section.  The first main result (Theorem \ref{thm:higher}) says that every multiplicative $n$-ary totally Hom-associative algebra yields a multiplicative $(2n-1)$-ary totally Hom-associative algebra.  The second main result (Theorem \ref{thm:lower}) says that under suitable conditions an $n$-ary totally Hom-associative algebra reduces to an $(n-1)$-ary totally Hom-associative algebra.  Both of these results can be iterated.

Here is the first main result of this section.  The Hom-Nambu analogue is discussed in \cite{yau13}.  Recall the abbreviations in \eqref{xij}.


\begin{theorem}
\label{thm:higher}
Let $(A,\nproduct,\alpha)$ be a multiplicative $n$-ary totally Hom-associative algebra.  Then
\[
A^1 = (A,\nproduct^{(1)},\alpha^2)
\]
is a multiplicative $(2n-1)$-ary totally Hom-associative algebra, where
\[
(a_{1,2n-1})^{(1)} = \left((a_{1,n}), \alpha(a_{n+1,2n-1})\right)
\]
for all $a_i \in A$.
\end{theorem}

\begin{proof}
It is clear that $A^1$ is a multiplicative $(2n-1)$-ary Hom-algebra. We must show that its $j$th Hom-associator $as^j_{A^1}$ (Definition \ref{def:nhomass}) is equal to $0$, that is,
\begin{equation}
\label{aoneass}
\begin{split}
&\left(\alpha^2(a_{1,j-1}),(a_{j,j+2n-2})^{(1)}, \alpha^2(a_{j+2n-1,4n-3})\right)^{(1)} \\
&= \left(\alpha^2(a_{1,j}),(a_{j+1,j+2n-1})^{(1)}, \alpha^2(a_{j+2n,4n-3})\right)^{(1)}
\end{split}
\end{equation}
for all $j \in \{1,\ldots,2n-2\}$.  The condition \eqref{aoneass} is divided into three cases: (1) $j \leq n-1$, (2) $j=n$, and (3) $j \geq n+1$.  Note that the case $j \geq n+1$ does not occur if $n=2$. The three cases are proved similarly, so we only provide the details for case (3), which is the only case where multiplicativity is used. With $j \geq n+1$, using the multiplicativity and total Hom-associativity of $A$, we compute the left-hand side of \eqref{aoneass} as follows:
\[
\begin{split}
&\left(\alpha^2(a_{1,j-1}),(a_{j,j+2n-2})^{(1)}, \alpha^2(a_{j+2n-1,4n-3})\right)^{(1)}\\
&= \left((\alpha^2(a_{1,n})), \alpha^3(a_{n+1,j-1}), \alpha\left((a_{j,j+n-1}), \alpha(a_{j+n,j+2n-2})\right), \alpha^3(a_{j+2n-1,4n-3})\right)\\
&= \left((\alpha^2(a_{1,n})), \alpha^3(a_{n+1,j-1}), \alpha\left(\alpha(a_j), (a_{j+1,j+n}), \alpha(a_{j+n+1,j+2n-2})\right), \alpha^3(a_{j+2n-1,4n-3})\right)\\
&= \left(\alpha^2((a_{1,n})), \alpha^3(a_{n+1,j-1}), \left(\alpha^2(a_j), \alpha(a_{j+1,j+n}), \alpha^2(a_{j+n+1,j+2n-2})\right), \alpha^3(a_{j+2n-1,4n-3})\right)\\
&= \left(\alpha^2((a_{1,n})), \alpha^3(a_{n+1,j}), \left(\alpha(a_{j+1,j+n}), \alpha^2(a_{j+n+1,j+2n-1})\right), \alpha^3(a_{j+2n,4n-3})\right)\\
&= \left((\alpha^2(a_{1,n})), \alpha^3(a_{n+1,j}), \alpha((a_{j+1,j+n}), \alpha(a_{j+n+1,j+2n-1})), \alpha^3(a_{j+2n,4n-3})\right)
\end{split}
\]
The last expression above is equal to the right-hand side of \eqref{aoneass}, as desired.
\end{proof}


Applying Theorem \ref{thm:higher} repeatedly, we obtain the following result.  It says that every multiplicative $n$-ary totally Hom-associative algebra gives rise to a sequence of multiplicative totally Hom-associative algebras of exponentially higher arities.

\begin{corollary}
\label{cor1:higher}
Let $(A,\nproduct,\alpha)$ be a multiplicative $n$-ary totally Hom-associative algebra.  Define the $(2^k(n-1)+1)$-ary product $\nproduct^{(k)}$ inductively by setting $\nproduct^{(0)} = \nproduct$ and
\[
\left(a_{1,2^k(n-1)+1}\right)^{(k)} = \left(\left(a_{1,2^{k-1}(n-1)+1}\right)^{(k-1)}, \alpha^{2^{k-1}}\left(a_{2^{k-1}(n-1)+2, 2^k(n-1)+1}\right)\right)^{(k-1)}
\]
for $k \geq 1$.  Then
\[
A^k = (A,\nproduct^{(k)},\alpha^{2^k})
\]
is a multiplicative $(2^k(n-1)+1)$-ary totally Hom-associative algebra for each $k \geq 0$.
\end{corollary}

For example, when $k=2$ we have:
\[
\begin{split}
\left(a_{1,4n-3}\right)^{(2)}
&= \left(\left(a_{1,2n-1}\right)^{(1)}, \alpha^2(a_{2n,4n-3})\right)^{(1)}\\
&= \left(\left(\left((a_{1,n}), \alpha(a_{n+1,2n-1})\right), \alpha^2(a_{2n,3n-2})\right), \alpha^3(a_{3n-1,4n-3})\right).
\end{split}
\]
When $k=3$, writing $x = \left(a_{1,4n-3}\right)^{(2)}$, we have:
\[
\begin{split}
\left(a_{1,8n-7}\right)^{(3)}
&= \left(x,\alpha^4(a_{4n-2,8n-7})\right)^{(2)}\\
&= \left(\left(\left((x,\alpha^4(a_{4n-2,5n-4})), \alpha^5(a_{5n-3,6n-5})\right), \alpha^6(a_{6n-4,7n-6})\right), \alpha^7(a_{7n-5,8n-7})\right).
\end{split}
\]
In general, $\nproduct^{(k)}$ involves an iterated composition of $2^k$ copies of the $n$-ary product $\nproduct$ and $n-1$ copies of $\alpha^i$ for each $i \in \{1,\ldots,2^k-1\}$.

Restricting to the case $\alpha = Id$ in Corollary \ref{cor1:higher}, we obtain the following construction result for higher arity totally associative algebras.

\begin{corollary}
\label{cor2:higher}
Let $(A,\nproduct)$ be an $n$-ary totally associative algebra.  Define the $(2^k(n-1)+1)$-ary product $\nproduct^{(k)}$ inductively by setting $\nproduct^{(0)} = \nproduct$ and
\[
\left(a_{1,2^k(n-1)+1}\right)^{(k)} = \left(\left(a_{1,2^{k-1}(n-1)+1}\right)^{(k-1)}, a_{2^{k-1}(n-1)+2, 2^k(n-1)+1}\right)^{(k-1)}
\]
for $k \geq 1$.  Then
\[
A^k = (A,\nproduct^{(k)})
\]
is a multiplicative $(2^k(n-1)+1)$-ary totally associative algebra for each $k \geq 0$.
\end{corollary}

The following result is the second main result of this section.  It gives sufficient conditions under which an $n$-ary totally Hom-associative algebra reduces to an $(n-1)$-ary totally Hom-associative algebra.  This result is the totally Hom-associative analogue of results due to Pozhidaev \cite{poz} and Filippov \cite{filippov} about $n$-ary Maltsev algebras and Nambu-Lie algebras, respectively.  The $n$-ary Hom-Nambu(-Lie) and Hom-Maltsev analogues can be found in \cite{yau13} and \cite{yau14}, respectively.

\begin{theorem}
\label{thm:lower}
Let $(A,\nproduct,\alpha)$ be an $n$-ary totally Hom-associative algebra with $n \geq 3$.  Suppose $a \in A$ satisfies
\begin{enumerate}
\item
$\alpha_{n-1}(a) = a$, and
\item
$(x_{1,n-1},a) = (x_{1,n-2},a,x_{n-1})$ for all $x_i \in A$.
\end{enumerate}
Then
\[
A_1 = (A,\nproduct',\alpha')
\]
is an $(n-1)$-ary totally Hom-associative algebra, where
\[
(x_{1,n-1})' = (x_{1,n-1},a) \quad\text{and}\quad
\alpha' = (\alpha_1,\ldots,\alpha_{n-2})
\]
for all $x_i \in A$.  Moreover, if $A$ is multiplicative, then so is $A_1$.
\end{theorem}

\begin{proof}
The multiplicativity assertion is clear.  We must show that for $j \in \{1,\ldots,n-2\}$, the $j$th Hom-associator $as^j_{A_1}$ is equal to $0$, that is,
\begin{equation}
\label{ahomass}
\begin{split}
&\left(\alpha_{1,j-1}(x_{1,j-1}), (x_{j,j+n-2})', \alpha_{j,n-2}(x_{j+n-1,2n-3})\right)'\\
&= \left(\alpha_{1,j}(x_{1,j}), (x_{j+1,j+n-1})', \alpha_{j+1,n-2}(x_{j+n,2n-3})\right)'.
\end{split}
\end{equation}
Using the assumptions on $a$ and the total Hom-associativity of $A$, we compute the left-hand side of \eqref{ahomass} as follows:
\[
\begin{split}
&\left(\alpha_{1,j-1}(x_{1,j-1}), (x_{j,j+n-2})', \alpha_{j,n-2}(x_{j+n-1,2n-3})\right)'\\
&= \left(\alpha_{1,j-1}(x_{1,j-1}), (x_{j,j+n-2},a), \alpha_{j,n-2}(x_{j+n-1,2n-3}), \alpha_{n-1}(a)\right)\\
&= \left(\alpha_{1,j}(x_{1,j}), (x_{j+1,j+n-2},a,x_{j+n-1}), \alpha_{j+1,n-2}(x_{j+n,2n-3}), \alpha_{n-1}(a)\right)\\
&= \left(\alpha_{1,j}(x_{1,j}), (x_{j+1,j+n-1},a), \alpha_{j+1,n-2}(x_{j+n,2n-3}), a\right).\\
\end{split}
\]
The last expression above is equal to the right-hand side of \eqref{ahomass}, as desired.
\end{proof}

Note that in Theorem \ref{thm:lower} the first assumption about $a$ is automatically satisfied if $\alpha_{n-1}$ is the identity map on $A$.  On the other hand, the second assumption on $a$ is automatically satisfied if the $n$-ary product $\nproduct$ is commutative in the last two variables.

Applying Theorem \ref{thm:lower} repeatedly, we obtain the following result.  It gives sufficient conditions under which an $n$-ary totally Hom-associative algebra reduces to an $(n-k)$-ary totally Hom-associative algebra.

\begin{corollary}
\label{cor1:lower}
Let $(A,\nproduct,\alpha)$ be an $n$-ary totally Hom-associative algebra with $n \geq 3$.  Suppose there exist $a_1,\ldots,a_k \in A$ for some $k \leq n-2$ such that
\begin{enumerate}
\item
$\alpha_{n-i}(a_i) = a_i$ for all $i \in \{1,\ldots,k\}$, and
\item
$(x_{1,n-i},a_i,a_{i-1},\ldots,a_1) = (x_{1,n-i-1},a_i,x_{n-i}, a_{i-1},\ldots,a_1)$ for all $i \in \{1,\ldots,k\}$ and $x_j \in A$.
\end{enumerate}
Then
\[
A_k = (A,\nproduct^k,(\alpha_1,\ldots,\alpha_{n-1-k}))
\]
is an $(n-k)$-ary totally Hom-associative algebra, where
\[
(x_{1,n-k})^k = (x_{1,n-k},a_k,a_{k-1},\ldots,a_1)
\]
for all $x_j \in A$.  Moreover, if $A$ is multiplicative, then so is $A_k$.
\end{corollary}

\section{From totally Hom-associative algebras to Hom-Nambu algebras}
\label{sec:homassnambu}

The main result of this section says that an $n$-ary totally Hom-associative algebra with equal twisting maps yields an $n$-ary Hom-Nambu algebra via the $n$-commutator bracket.

Let us first define the $n$-commutator words, which generalize the two terms in the usual commutator bracket.

\begin{definition}
\label{def:ncommutator}
Let $X_1,X_2,\ldots$ be non-commuting variables.  For $n \geq 2$ define the set $W_n$ of \textbf{$n$-commutator words} inductively as
\[
W_2 = \{X_1X_2, -X_2X_1\}
\]
and
\[
W_n = \{zX_n, -X_nz \colon z \in W_{n-1}\}
\]
for $n > 2$.
\end{definition}

For example,
\[
W_3 = \{X_1X_2X_3, -X_2X_1X_3, -X_3X_1X_2, X_3X_2X_1\}
\]
and $W_4$ consists of the $4$-commutator words
\[
\begin{split}
& X_1X_2X_3X_4, -X_2X_1X_3X_4, -X_3X_1X_2X_4, X_3X_2X_1X_4,\\
& -X_4X_1X_2X_3, X_4X_2X_1X_3, X_4X_3X_1X_2, -X_4X_3X_2X_1.
\end{split}
\]
In general, $W_n$ consists of $2^{n-1}$ $n$-commutator words.  Every $n$-commutator word gives a self-map on the $n$-fold tensor product as follows.

\begin{definition}
\label{def:nmap}
Let $V$ be a $\bk$-module and $w = \pm X_{i_1}\cdots X_{i_n} \in W_n$.  Define the map $w \colon V^{\otimes n} \to V^{\otimes n}$ by
\[
w(v_1,\ldots,v_n) = \pm v_{i_1} \otimes \cdots \otimes v_{i_n}
\]
for all $v_j \in V$.
\end{definition}

Using these maps defined by the $n$-commutator words, we can now define the $n$-commutator bracket.

\begin{definition}
\label{def:nbracket}
Let $(A,\nproduct,\alpha)$ be an $n$-ary Hom-algebra.  Define the \textbf{$n$-commutator bracket} $\bracket \colon A^{\otimes n} \to A^{\otimes n}$ by
\begin{equation}
\label{ncomm}
[a_1,\ldots,a_n] = \sum_{w \in W_n} \left(w(a_1,\ldots,a_n)\right)
\end{equation}
for all $a_j \in A$.
\end{definition}

For example, the $2$-commutator bracket is the usual commutator bracket:
\[
[a_1,a_2] = (a_1a_2) - (a_2a_1).
\]
The $3$-commutator bracket is the sum
\[
[a_{1,3}] = (a_1a_2a_3) - (a_2a_1a_3) - (a_3a_1a_2) + (a_3a_2a_1),
\]
which was first used in Corollary 4.3 in \cite{yau13}.

Let us now recall the definition of an $n$-ary Hom-Nambu algebra from \cite{ams}.

\begin{definition}
\label{def:homnambulie}
Let $(V,\bracket,\alpha)$ be an $n$-ary Hom-algebra.
\begin{enumerate}
\item
The \textbf{$n$-ary Hom-Jacobian} of $V$ is the $(2n-1)$-linear map $J^n_V \colon V^{\otimes 2n-1} \to V$ defined as (using the shorthand in \eqref{xij})
\begin{equation}
\label{homjacobian}
\begin{split}
J^n_V(x_{1,n-1};y_{1,n})
&= [\alpha_{1,n-1}(x_{1,n-1}), [y_{1,n}]]\\
&\relphantom{} - \sum_{i=1}^n \, [\alpha_{1,i-1}(y_{1,i-1}), [x_{1,n-1},y_i], \alpha_{i,n-1}(y_{i+1,n})]
\end{split}
\end{equation}
for $x_1, \ldots, x_{n-1}, y_1, \ldots , y_n \in V$.
\item
An \textbf{$n$-ary Hom-Nambu algebra} is an $n$-ary Hom-algebra $V$ that satisfies
\[
J^n_V = 0,
\]
called  the \textbf{$n$-ary Hom-Nambu identity}.
\end{enumerate}
\end{definition}

When the twisting maps are all equal to the identity map, $n$-ary Hom-Nambu algebras are the usual \textbf{$n$-ary Nambu algebras} \cite{ai,filippov,nambu,plet}.  In this case, the $n$-ary Hom-Jacobian $J^n_V$ is called the \textbf{$n$-ary Jacobian}, and the $n$-ary Hom-Nambu identity $J^n_V = 0$ is called the \textbf{$n$-ary Nambu identity}.


We are now ready for the main result of this section.

\begin{theorem}
\label{thm:commutator}
Let $(A,\nproduct,\alpha)$ be an $n$-ary totally Hom-associative algebra with equal twisting maps.  Then
\[
N(A) = (A,\bracket,\alpha)
\]
is an $n$-ary Hom-Nambu algebra, where $\bracket$ is the $n$-commutator bracket in \eqref{ncomm}.  Moreover, if $A$ is multiplicative, then so is $N(A)$.
\end{theorem}

Since the $n$-commutator bracket does not involve the twisting maps, we have the following special case of Theorem \ref{thm:commutator}

\begin{corollary}
\label{cor:commutator}
Let $(A,\nproduct)$ be an $n$-ary totally associative algebra with equal twisting maps.  Then
\[
N(A) = (A,\bracket)
\]
is an $n$-ary Nambu algebra, where $\bracket$ is the $n$-commutator bracket.
\end{corollary}

For the proof of Theorem \ref{thm:commutator}, we need to prove the $n$-ary Hom-Nambu identity $J^n_{N(A)} = 0$.  In the following Lemmas, we first compute the various terms in the $n$-ary Hom-Jacobian $J^n_{N(A)}$ \eqref{homjacobian}.  The $n$ terms in the sum in \eqref{homjacobian} are considered in two cases, $i=n$ (Lemma \ref{lemBn}) and $1 \leq i \leq n-1$ (Lemma \ref{lemBi}).

Let us compute the first term in the $n$-ary Hom-Jacobian $J^n_{N(A)}$.

\begin{lemma}
\label{lemA}
With the hypotheses of Theorem \ref{thm:commutator}, we have
\begin{equation}
\label{A}
\begin{split}
[\alpha(x_{1,n-1}),[y_{1,n}]]
&= \sum_{z,z' \in W_{n-1}} \left(z'(\alpha(x_{1,n-1})), (z(y_{1,n-1}), y_n)\right)\\
&\relphantom{} - \sum_{z,z' \in W_{n-1}} \left(z'(\alpha(x_{1,n-1})), (y_n, z(y_{1,n-1}))\right)\\
&\relphantom{} - \sum_{z,z' \in W_{n-1}} \left((z(y_{1,n-1}), y_n), z'(\alpha(x_{1,n-1}))\right)\\
&\relphantom{} + \sum_{z,z' \in W_{n-1}} \left((y_n, z(y_{1,n-1})), z'(\alpha(x_{1,n-1}))\right)
\end{split}
\end{equation}
for all $x_j, y_l \in A$.
\end{lemma}

\begin{proof}
From the definitions of the $n$-commutator bracket \eqref{ncomm} and the $n$-commutator words (Definition \ref{def:ncommutator}), we have
\begin{equation}
\label{bracket}
\begin{split}
[a_{1,n}] &= \sum_{w \in W_n} (w(a_{1,n}))\\
&= \sum_{z \in W_{n-1}} (z(a_{1,n-1}), a_n) - (a_n, z(a_{1,n-1})).
\end{split}
\end{equation}
for all $a_j \in A$.  The Lemma is obtained by using \eqref{bracket} on both $[y_{1,n}]$ and $[\alpha(x_{1,n-1}),\cdot]$.
\end{proof}

In Lemma \ref{lemA} we did not use the total Hom-associativity of $A$.  In the next result, we compute the $i=n$ term in the sum in the $n$-ary Hom-Jacobian $J^n_{N(A)}$.

\begin{lemma}
\label{lemBn}
With the hypotheses of Theorem \ref{thm:commutator}, we have
\begin{equation}
\label{Bn}
\begin{split}
-[\alpha(y_{1,n-1}), [x_{1,n-1},y_n]]
&= -\sum_{z,z' \in W_{n-1}} \left(z(\alpha(y_{1,n-1})), (z'(x_{1,n-1}), y_n)\right)\\
& \relphantom{} +\sum_{z,z' \in W_{n-1}} \left((z(y_{1,n-1}), y_n), z'(\alpha(x_{1,n-1}))\right)\\
&\relphantom{} + \sum_{z,z' \in W_{n-1}} \left(z'(\alpha(x_{1,n-1})), (y_n, z(y_{1,n-1}))\right)\\
&\relphantom{} - \sum_{z,z' \in W_{n-1}} \left((y_n, z'(x_{1,n-1})), z(\alpha(y_{1,n-1}))\right)
\end{split}
\end{equation}
for all $x_j, y_l \in A$.
\end{lemma}

\begin{proof}
Using \eqref{bracket} we have
\begin{equation}
\label{Bn'}
\begin{split}
-[\alpha(y_{1,n-1}), [x_{1,n-1},y_n]]
&= -\sum_{z,z' \in W_{n-1}} \left(z(\alpha(y_{1,n-1})), (z'(x_{1,n-1}), y_n)\right)\\
& \relphantom{} +\sum_{z,z' \in W_{n-1}} \left(z(\alpha(y_{1,n-1})), (y_n, z'(x_{1,n-1}))\right)\\
&\relphantom{} + \sum_{z,z' \in W_{n-1}} \left((z'(x_{1,n-1}), y_n), z(\alpha(y_{1,n-1}))\right)\\
&\relphantom{} - \sum_{z,z' \in W_{n-1}} \left((y_n, z'(x_{1,n-1})), z(\alpha(y_{1,n-1}))\right).
\end{split}
\end{equation}
The first (resp., fourth) sums in \eqref{Bn} and \eqref{Bn'} are equal.  The second (resp., third) sums in \eqref{Bn} and \eqref{Bn'} are equal by the total Hom-associativity of $A$.
\end{proof}

Combining Lemmas \ref{lemA} and \ref{lemBn}, we obtain the following result.

\begin{lemma}
\label{lemABn}
With the hypotheses of Theorem \ref{thm:commutator}, we have
\[
\begin{split}
[\alpha(x_{1,n-1}),[y_{1,n}]] & - [\alpha(y_{1,n-1}), [x_{1,n-1},y_n]]\\
&= \sum_{z,z' \in W_{n-1}} \underbrace{\left(z'(\alpha(x_{1,n-1})), (z(y_{1,n-1}), y_n)\right)}_{A_1}\\
&\relphantom{} + \sum_{z,z' \in W_{n-1}} \underbrace{\left((y_n, z(y_{1,n-1})), z'(\alpha(x_{1,n-1}))\right)}_{A_4}\\
&\relphantom{} - \sum_{z,z' \in W_{n-1}} \underbrace{\left(z(\alpha(y_{1,n-1})), (z'(x_{1,n-1}), y_n)\right)}_{B_{n1}}\\
&\relphantom{} - \sum_{z,z' \in W_{n-1}} \underbrace{\left((y_n, z'(x_{1,n-1})), z(\alpha(y_{1,n-1}))\right)}_{B_{n4}}
\end{split}
\]
for all $x_j, y_l \in A$.
\end{lemma}

\begin{proof}
This follows from Lemmas \ref{lemA} and \ref{lemBn} because the second (resp., third) sum in \eqref{A} is equal to the third (resp., second) sum in \eqref{Bn} with the opposite sign.
\end{proof}

The labels $A_1$, $A_4$, $B_{n1}$, and $B_{n4}$ will be used below.

Next we compute the $i$th term ($1 \leq i \leq n-1$) in the sum in the $n$-ary Hom-Jacobian $J^n_{N(A)}$.

\begin{lemma}
\label{lemBi}
With the hypotheses of Theorem \ref{thm:commutator}, for $i \in \{1,\ldots,n-1\}$ we have
\[
\begin{split}
- & [\alpha(y_{1,i-1}), [x_{1,n-1},y_i], \alpha(y_{i+1,n})]\\
&= - \sum_{z,z' \in W_{n-1}} \underbrace{\left(z(\alpha(y_{1,i-1}), (z'(x_{1,n-1}), y_i), \alpha(y_{i+1,n-1})), \alpha(y_n)\right)}_{B_{i1}}\\
&\relphantom{} + \sum_{z,z' \in W_{n-1}} \underbrace{\left(\alpha(y_n), z(\alpha(y_{1,i-1}), (z'(x_{1,n-1}), y_i), \alpha(y_{i+1,n-1}))\right)}_{B_{i2}}\\
&\relphantom{} + \sum_{z,z' \in W_{n-1}} \underbrace{\left(z(\alpha(y_{1,i-1}), (y_i, z'(x_{1,n-1})), \alpha(y_{i+1,n-1})), \alpha(y_n)\right)}_{B_{i3}}\\
&\relphantom{} - \sum_{z,z' \in W_{n-1}} \underbrace{\left(\alpha(y_n), z(\alpha(y_{1,i-1}), (y_i, z'(x_{1,n-1})), \alpha(y_{i+1,n-1}))\right)}_{B_{i4}}
\end{split}
\]
for all $x_j, y_l \in A$.
\end{lemma}

\begin{proof}
Just use \eqref{bracket} twice, as in Lemma \ref{lemA}.
\end{proof}

Using Lemmas \ref{lemABn} and \ref{lemBi} we can now give the proof of Theorem \ref{thm:commutator}.

\begin{proof}[Proof of Theorem \ref{thm:commutator}]
Since the multiplicativity assertion is clear, it remains to establish the $n$-ary Hom-Nambu identity $J^n_{N(A)} = 0$.  By Lemmas \ref{lemABn} and \ref{lemBi} the $n$-ary Hom-Jacobian $J^n_{N(A)}$ is
\begin{equation}
\label{jna}
J^n_{N(A)} = \sum_{z,z'\in W_{n-1}} \left(A_1 + A_4 - B_{n1} - B_{n4} + \sum_{i=1}^{n-1}(-B_{i1} + B_{i2} + B_{i3} - B_{i4})\right).
\end{equation}
The total Hom-associativity of $A$ implies the following six types of cancellation.  For $i \in \{1,\ldots,n-1\}$ we have:
\[
\sum_{\substack{z' \in W_{n-1}\\ z = \pm X_i \cdots \in W_{n-1}}} (A_1 - B_{i1}) = 0,
\quad
\sum_{\substack{z' \in W_{n-1}\\ z = \pm \cdots X_i \in W_{n-1}}} (A_4 - B_{i4}) = 0,
\]
\[
\sum_{\substack{z' \in W_{n-1}\\ z = \pm \cdots X_i \in W_{n-1}}} (-B_{n1} + B_{i3}) = 0, \quad
\sum_{\substack{z' \in W_{n-1}\\ z = \pm X_i \cdots \in W_{n-1}}}  (-B_{n4} + B_{i2}) = 0.
\]
For $i \not= j$ in $\{1,\ldots,n-1\}$ we have:
\[
\sum_{\substack{z' \in W_{n-1}\\ z = \pm \cdots X_jX_i \cdots \in W_{n-1}}} (-B_{i1} + B_{j3}) = 0,\quad
\sum_{\substack{z' \in W_{n-1}\\ z = \pm \cdots X_jX_i \cdots \in W_{n-1}}} (B_{i2} - B_{j4}) = 0.
\]
Here $z = \pm X_i \cdots \in W_{n-1}$ (resp., $z = \pm \cdots X_i \in W_{n-1}$) means that $z$ is an $(n-1)$-commutator word starting (resp., ending) with $X_i$.  Likewise, $z = \pm \cdots X_jX_i \cdots \in W_{n-1}$ means that $z$ is an $(n-1)$-commutator word in which $X_i$ is immediately preceded by $X_j$.  Using the expression \eqref{jna} for $J^n_{N(A)}$, these six types of cancellation imply that $N(A)$ satisfies the $n$-ary Hom-Nambu identity $J^n_{N(A)} = 0$.
\end{proof}


\end{document}